\documentclass[12pt,a4paper]{amsart}

\usepackage{amssymb}
\usepackage{amscd}
\usepackage{hyperref}
\usepackage{stmaryrd}
\usepackage[capitalise]{cleveref}
\usepackage{tikz}
\usepackage{tikz-cd} 
\usepackage{esvect}
\theoremstyle{plain}
\newtheorem{thm}{Theorem}[section]
\newtheorem{lem}[thm]{Lemma}
\crefname{lem}{Lemma}{Lemmas}

\newtheorem{cor}[thm]{Corollary}

\theoremstyle{definition}
\newtheorem{dfn}[thm]{Definition}

\theoremstyle{remark}
\newtheorem{rem}[thm]{Remark}

\newcommand{\R}{\mathbb{R}}
\newcommand{\Q}{\mathbb{Q}}
\newcommand{\Z}{\mathbb{Z}}

\newcommand{\cC}{\mathcal{C}}

\newcommand{\cvec}{\mathbf{Vec}}

\DeclareMathOperator{\Hom}{Hom}

\DeclareMathOperator{\Ker}{Ker}
\DeclareMathOperator{\Coker}{Coker}
\DeclareMathOperator{\Ima}{Im}

\DeclareMathOperator{\End}{End}

\newcommand{\sdb}{{\textbf{db}}}
\newcommand{\sbb}{\textbf{bb}}
\newcommand{\svb}{{\textbf{vb}}}
\newcommand{\shb}{\textbf{hb}}
\newcommand{\FS}{\mathcal S^\uparrow}
\newcommand{\FI}[1]{\mathcal S(#1)}
\newcommand{\HH}{{\rm H}}

\begin{document}
\title{Decomposition of persistence modules}

\author{Magnus Bakke Botnan}
\address{Department of Mathematics, VU University Amsterdam, The Netherlands}
\email{m.b.botnan@vu.nl}

\author{William Crawley-Boevey}
\address{Fakult\"at f\"ur Mathematik, Universit\"at Bielefeld, 33501 Bielefeld, Germany}
\email{wcrawley@math.uni-bielefeld.de}



\thanks{The first author has been supported by the DFG Collaborative Research Center SFB/TR 109 “Discretization in Geometry and Dynamics”. The second author has been supported by the Alexander von Humboldt Foundation in the framework of an Alexander von Humboldt Professorship endowed by the German Federal Ministry of Education and Research.}

\begin{abstract}
We show that a pointwise finite-dimensional persistence module indexed over a small category decomposes into a direct sum of indecomposables with local endomorphism rings. As an application of this result we give new, short proofs of fundamental structure theorems for persistence modules. 
\end{abstract}
\maketitle

\section{Introduction}
Let $\cC$ be a small category and write $\cvec$ for the category of vector spaces over a field $k$.
By a \emph{persistence module} (over $\cC$) we mean a functor $M\colon \cC\to\cvec$. We say that $M$ is \emph{pointwise finite-dimensional} if each $M_x$ is finite-dimensional.

The work in this paper is inspired by topological data analysis (TDA). For an introduction to TDA, see e.g. the survey by Carlsson \cite{carlsson2009topology}, or the recent book 
by Oudot \cite{oudot2015persistence} on quiver representations and TDA.   

Let $X$ be a topological space, $h\colon X\to \R$ a continuous function, and consider the following functors 
\begin{align*}
\FS(h)\colon& \R\to {\rm Top} \quad & \FS(h)(t) = \{x\in X \mid h(x) \leq t\} \\
 \FI{h}\colon& \R^2\to {\rm Top} \quad & \FI{h}(-s,t) = \{x\in X \mid s< h(x) < t\}
\end{align*}
\emph{Persistent homology} studies the evolution of the homology of the sublevel sets of $h$ and is perhaps the most prominent tool in TDA. Specifically, the \emph{$p$-th sublevel set persistence module associated to $h$} is the functor $\HH_p\FS(h)\colon \R\to \cvec$. Here $\HH_p\colon {\rm Top} \to \cvec$ denotes the $p$-th singular homology functor with coefficients in $k$. Importantly, and as we shall see later in this paper, if $\HH_p\FS(h)$ is pointwise finite-dimensional, then it is completely determined by a collection of intervals called the \emph{barcode} of $\HH_p\FS(h)$. This collection of intervals is then in turn used to extract topological information from the data at hand; a ''long'' interval corresponds to a topological feature which persists over a significant range. A richer invariant is obtained by considering interlevel sets: define the \emph{$p$-th interlevel set persistence of $h$} to be the functor $\HH_p \FI{h}\colon \R^2\to \cvec$. By a Mayer-Vietoris argument \cite{cochoy2016decomposition} one can show that $\HH_p\FI{h}$ is middle exact (see \cref{sec:upper}) when restricted to the points above the anti-diagonal. Analogously to above,  assuming that $\HH_p\FI{h}$ is pointwise finite-dimensional, such a module is completely determined by a collection of simple regions in $\R^2$. These regions in turn give valuable insight into the homological properties of the fibers of the function $h$.  We refer the reader to \cite{botnan2016algebraic,cochoy2016decomposition} for an in-depth treatment. 

We also remark that there are many settings for which it is fruitful to combine a collection of real-valued functions into a single function $g\colon X\to \R^n$ \cite{carlsson2009theory}. By combining them into a single function we not only learn how the data looks from the point of view of each function (i.e. a type of measurement) but how the different functions (measurements) interact. How to effectively use such persistence modules in data analysis is not clear and for the time being an area of active research, see e.g. \cite{lesnick2015interactive} and the references therein.

\subsection{Contributions}
We give a short direct proof of the following result. 
\begin{thm}
\label{t:decomp}
Any pointwise finite-dimensional persistence module is a direct sum of indecomposable modules with local endomorphism ring.
\end{thm}
We remark that this result is already known by the theory of locally finitely presented additive categories.
The category $\cvec$ is locally finitely presented, hence so is the category of persistence modules,
which is a functor category.  Now any pointwise finite-dimensional module is a direct sum of
indecomposables with local endomorphism ring by the theory of $\Sigma$-pure-injectives,
see (3)$\Rightarrow$(4) of \cite[\S3.2 Theorem 2]{CBlfp}.

Persistence modules are often considered for partially ordered sets (where $\cC$ is the naturally associated category).
Using this result, we give a short proof of the following result, originally proved in a slightly weaker form in \cite{CBdpf}.

\begin{thm}
\label{t:totorder}
Pointwise finite-dimensional persistence modules over a totally ordered set decompose into interval modules.
\end{thm}
Note that the advantage of the approach in \cite{CBdpf} is that it produces functors 
which give the multiplicity of any interval module as a direct summand. 

Following the ideas of \cite{CBdpf}, \cref{t:totorder} was generalized to exact (middle exact in this paper) bi-modules in \cite{cochoy2016decomposition}. We give a comparatively short proof of a slight generalization of the main theorem of \cite{cochoy2016decomposition}.

\begin{thm}
Pointwise finite-dimensional middle exact modules over a product of two totally ordered sets decompose into block modules.
\label{thm:block}
\end{thm}
As a corollary to this we obtain a structure theorem for pointwise finite-dimensional persistence modules on \emph{zigzag paths}. This generalizes the structure theorem for \emph{zigzag persistent homology} given in \cite{botnanzigzag}. We refer the reader to \cite{botnanzigzag} and the references therein for a discussion on zigzag persistent homology. In the last part of the paper we apply the structure theorem for persistence modules indexed by zigzag paths to prove a structure theorem for persistence modules that are middle exact (strictly) above the anti-diagonal in $\R^2$. 
\begin{thm}
Pointwise finite-dimensional middle exact modules over the (strictly) upper-triangular subset of the plane decompose into block modules.\label{thm:upperT}
\end{thm} 

\begin{rem}
We are indebted to D. Vossieck for pointing out the reference \cite[\S 3.6]{gabriel1992representations},
where Theorems \ref{t:decomp} and \ref{t:totorder} are both discussed, with sketch proofs. As
this paper is intended to be self-contained and aimed at a broader
audience, we include detailed proofs of both theorems. In fact, our proof of \cref{thm:block} depends in part on \cref{t:totorder}, and, as the reader will see, very little work is needed to prove \cref{t:totorder} once the machinery for proving (the complementary parts of) \cref{thm:block} has been introduced. 
\end{rem}

\section{Preliminaries}
Let $\cC$ be a small category and $M,N\colon \cC\to \cvec$. If $x$ is an object in $\cC$ we write $M_x$ for the corresponding vector space, and if $\alpha\colon x\to y$ is
a morphism, we write $M_\alpha\colon M_x\to M_y$. A morphism $f\colon M\to N$ is an \emph{epimorphism} (\emph{monomorphism}) if $f_x\colon M_x\to N_x$ is surjective (injective) for all $x\in {\rm Ob}(\cC)$. A morphism is an \emph{isomorphism} if it is both an epimorphism and a monomorphism. A monomorphism $f\colon M\to N$ \emph{splits}, or is a \emph{split monomorphism}, if there exists a $g\colon N\to M$ such that $g\circ f = {\rm id}_M$. We say that $M$ and $N$ are \emph{isomorphic} if there exists an isomorphism $f\colon M\to N$ and denote this by $M\cong N$. The \emph{direct sum} of $M$ and $N$ is the persistence module $M\oplus N\colon \cC\to \cvec$ given by $(M\oplus N)_x = M_x\oplus N_x$ and $(M\oplus N)_\alpha = M_\alpha\oplus N_\alpha$ for all $\alpha\colon x\to y$. The persistence module $M'$ is a \emph{submodule} of $M$ if $M'_x \subseteq M_x$ and $M'_\alpha$ is the restriction of $M_\alpha$ to $M'_x$ for all $\alpha\colon x\to y$. We write $M'\subseteq M$ if $M'$ is a submodule of $M$. If $M$ has two non-trivial submodules $M'$ and $M''$ such that $M=M'\oplus M'$, then $M$ is \emph{decomposable} and $M'$ and $M''$ are \emph{summands} of $M$. If no such decomposition exists, then $M$ is \emph{indecomposable}. It is an elementary fact that $M'\subseteq M$ is a summand of $M$ if and only the inclusion $M'\hookrightarrow M$ splits. If every monomorphism with domain $M$ splits, then $M$ is an \emph{injective persistence module}. 

The endomorphism ring $\End(M):= \Hom(M,M)$ is \emph{local} if $\theta$ or $1-\theta$ is invertible for all $\theta\in \End(M)$. The Krull--Remak--Schmidt--Azumaya theorem\cite{azumaya} asserts that persistence modules which decompose into a direct sum of indecomposables with a local endomorphism ring, do so in an essentially unique way (unique up to reordering and isomorphism). If $M$ has a non-trivial decomposition then $\End(M)$ is not local. 

Dualizing each vector space and each linear map in a persistence module $M\colon \cC \to \cvec$ yields a persistence module $DM\colon \cC^{\rm op}\to \cvec$. Here $\cC^{\rm op}$ denotes the opposite category of $\cC$. This dualization procedure is contravariantly functorial, exact and satisfies $D^2M\cong M$ whenever $M$ is pointwise finite-dimensional. 
\subsection{Posets} Let $P$ be a partially ordered set (poset). Recall that $P$ can be considered as a category with objects the elements of $P$ in a natural way:
\[
\Hom(p,q) = \begin{cases}
\{ \iota_{qp} \} & (p\le q) \\
\varnothing & (p \not\le q)
\end{cases}
\]
If $Q\subseteq P$ and $M\colon P\to \cvec$, then $M|_Q$ denotes the restriction of $M$ to $Q$.  A subset $I\subseteq P$ is \emph{convex} if $p\leq q \leq r$ with $p,r\in P$ implies that $q\in P$. If $I$ satisfies the stronger condition that $q\in I$ whenever $q\leq p$ and $p\in I$, then we say that $I$ is an \emph{ideal}. Dually, if $I$ satisfies that $q\in I$ whenever $q\geq p$ and $p\in I$, then we say that $I$ is a \emph{filter}. Furthermore, $I$ is \emph{connected} if there for every $p,q\in P$ exists a sequence $\{r_i\}_{i=0}^u\subseteq I$  such that $r_0 =p$, $r_u = q$ and $r_i\leq r_{i+1}$ or $r_i\geq r_{i+1}$ for all $0\leq i<u$. We define an \emph{interval} to be a non-empty, connected, and convex set. Examples of intervals include $[p,q]$, $(p,q)$, $(p,q]$ and $(p,q)$, where $[p,q] = \{r\in P \mid p\leq r \leq q\}$, and similarly for the other cases. We also have intervals $[p,\infty)= \{r\in P \mid r\geq p\}$ and $(p, \infty) = \{r\in P\mid r>p\}$, and similarly for $(-\infty, p)$ and $(-\infty, p]$. The notation $\langle p,q\rangle$ is used to denote any of the appropriate intervals in $\{(p,q), [p,q), (p,q], [p,q]\}$. E.g, we have $\langle p ,\infty\rangle \in \{ (p,\infty), [p,\infty)\}$. When $P$ is totally ordered the intervals are precisely the non-empty convex sets, and if $P = \R$, they are all of the form $\langle p,q\rangle$. Observe that the subset $\{x \mid x^2<2\}\subseteq \Q$ is an interval which is not of the form $\langle p,q\rangle$. 

For an interval $I\subseteq P$, we write $k_I$ for the \emph{constant module} which is 1-dimensional at points on $I$, zero at points outside $I$, and with the the morphisms $\iota_{yx}$ for $x,y\in I$ sent to the identity map. It follows from $\End(k_I)\cong k$ that $k_I$ is indecomposable \cite[Proposition 2.2]{botnan2016algebraic}.

A subset $I\subseteq P$ is \emph{directed} if there for every $p,q\in I$ exists a $c\in I$ satisfying $p,q\leq c$.  Dually, $I\subseteq P$ is \emph{codirected} if there for every $p,q\in I$ exists a $c\in I$ satisfying $p,q\geq c$. 
\begin{lem}
Let $I\subseteq P$ be a directed ideal. Then $k_I\colon P\to \cvec$ is an injective persistence module.
\label{l:injective}
\end{lem}
\begin{proof}
This follows from the fact that $\varinjlim_{p\in I}$ is an exact functor whenever $I$ is directed. Assume that $f\colon k_I \hookrightarrow M$ is a monomorphism and consider its restriction 
to $I$, $f|_I\colon (k_I)|_I \hookrightarrow M|_I$. By the aforementioned exactness property \[\hat{f}:=\varinjlim_{p\in I} f|_I\colon \varinjlim_{p\in I} (k_I)|_I  \hookrightarrow \varinjlim_{p\in I} M|_I\] is an injection. Let $\hat{g}$ be a left inverse to $\hat{f}$ and for $p\in I$ define $g_p: M_p\to (k_I)_p$ as the composition 
\[M_p \to \varinjlim_{p\in I} M|_I \xrightarrow{\hat{g}} \varinjlim_{p\in I} (k_I)|_I \xrightarrow{\cong} (k_I)_p = k.\]
For $p\not\in I$ define $g_p = 0$. It is clear that $g\circ f = {\rm id}_{k_I}$. 
\end{proof}
We remark that the converse statement of the previous lemma is also true \cite[Proposition 1.1]{hoppner1983note}. 

\begin{lem}
\label{lem:codirected}
Suppose $P$ is codirected. Let M be a pointwise finite-dimensional persistence module over $P$ with $M_p \neq 0$ for all $p\in P$, and suppose that $M_p \to M_q$ is injective for all $p \leq q$. Then there is a monomorphism $k_P\hookrightarrow M$. In particular, if P is also directed, then M has a copy of $k_P$ as a direct summand. 
\end{lem}
\begin{proof}
Let $p$ be a point such that $M_p$ is of minimal dimension, and choose a non-zero element $m_p$ in $M_p$. For any other point $q$ in $P$, there is an element $c$ with $p,q\geq c$. Since $\dim M_c = \dim M_p$, the morphism $M_c\to M_p$ is an isomorphism. Thus $m_p$ induces an element $M_q = M_{\iota_{qc}}(M_{\iota_{pc}}^{-1}(m_p)) $ in $M_q$. Using the codirectedness property again, it is easy to check that this
does not depend on the choice of $c$, and that the elements $m_q$ define a
morphism $k_P \to M$. This yields a monomorphism $k_P\hookrightarrow M$. The last part is immediate from \cref{l:injective}. 
\end{proof}
The following dual result will be important. 
\begin{lem}
Suppose $P$ is directed and codirected. Let $M$ be a pointwise finite-dimensional persistence
module over $P$ with $M_p\neq 0$ for all $p\in P$, and suppose that $M_p \to M_q$ is an epimorphism for all $p\leq q$. Then $M$ has a copy of $k_P$ as a direct summand.
\label{l:directecodirected}
\end{lem}
\begin{proof}
Observe that $P^{{\rm op}}$ is both directed and codirected. It follows that $DM\colon P^{{\rm op}}\to \cvec$ satisfies the conditions of \cref{lem:codirected}. Hence $DM$ has a copy of $k_{P^{{\rm op}}}$ as a direct summand. Using that $M$ is pointwise finite-dimensional we get that $M\cong DDM$ has a copy of $D(k_{P^{{\rm op}}}) \cong k_P$ as a direct summand. 
\end{proof}


\section{Decomposition}
In this section we prove Theorem~\ref{t:decomp}.
Our argument is inspired by Ringel's proof of the corresponding result for covering functors, see \cite{RingelIzmir}.

First suppose $M$ is a pointwise finite-dimensional indecomposable module, and let $\theta$ be an endomorphism.
If $x$ is an object in $\cC$ then $\theta$ induces an endomorphism $\theta_x$ of $M_x$.
Since $M_x$ is finite-dimensional, Fitting's lemma gives a decomposition
\[
M_x = M'_x \oplus M''_x
\]
where $M'_x = \Ima(\theta_x^n)$ for $n\gg 0$ and $M''_x = \Ker(\theta_x^n)$ for $n\gg 0$.
Moreover $\theta_x$ induces an automorphism of $M'_x$ and a nilpotent endomorphism of $M''_x$.

Now if $\alpha:x\to y$ is a morphism in $\cC$ then $M_\alpha \theta_x = \theta_y M_\alpha$.
Moreover $M_\alpha$ sends $M'_x$ into $M'_y$ and $M''_x$ into $M''_y$.
Namely, taking $n$ to be sufficiently large for the decompositions of $M_x$ and $M_y$, we
have $M_\alpha \theta_x^n = \theta_y^n M_\alpha$, so if $m\in M''_x = \Ker(\theta_x^n)$,
then $\theta_y^n M_\alpha(m)=0$, so $M_\alpha(m)\in \Ker(\theta_y^n) = M''_y$.
If $m\in M'_x$ then $m = \theta_x^n(m')$, so $M_\alpha(m) = \theta_y^n M_\alpha(m') \in \Ima(\theta_y^n) = M'_y$.

It follows that the decomposition $M_x = M'_x\oplus M''_x$ for each object $x$ in $\cC$
gives a decomposition of $M = M' \oplus M''$ as a
persistence module. Thus if $M$ is indecomposable, $M = M'$ or $M = M''$.
In the first case $\theta_x$ is invertible for all $x$, so $\theta$ is invertible.

If $\theta$ is not invertible, then the above decomposition shows that $\theta_x$ is nilpotent for all $x$. Assume that $(1-\theta_x)(m) = 0$ for $m\neq 0$ and
let $n\geq 2$ be the smallest integer such that $\theta^n_x(m) = 0$. Then $\theta^{n-1}_x\circ (1-\theta_x)(m) = \theta^{n-1}(m) =0$, contradicting that
$n$ was the minimal such $n$. Thus $\ker (1-\theta_x) = 0$ and $1-\theta$ is invertible for all $x$. We conclude that $\End(M)$ is local. 


Now let $M$ be a non-zero pointwise finite-dimensional persistence module, 
and let $D$ be the set of decompositions of $M$ into a direct sum of non-zero submodules.
That is, letting $S$ be the set of non-zero submodules of $M$,
$D$ is the set of subsets $I$ of $S$ such that $M = \bigoplus_{N\in I} N$.
We consider the relation $\le$ on $D$ with $I \le J$ if $J$ is a refinement of $I$. That is, if
each element of $J$ is contained in an element of $I$, or equivalently 
if each $N\in I$ is a direct sum of a subset of elements of $J$.
In this case there is a uniquely determined mapping $f_{IJ}:J\to I$ such that
for $N\in I$ we have 
\[
N = \bigoplus_{L\in f_{IJ}^{-1}(N)} L.
\]
Moreover $f_{IJ}$ is clearly surjective.
It is easy to see that this relation $\le$ defines a partial ordering on $D$.
Clearly $D$ is non-empty since it contains the element $\{M\}$ (as a unique minimal element).

To prove the theorem, it suffices to prove that $D$ contains a maximal element, for 
if $I\in D$ and $N\in I$ is decomposable, say $N = N_1\oplus N_2$, then
$J = (I \setminus \{N\})\cup \{N_1,N_2\}$ is in $D$, and $I < J$.
Thus if $I$ is a maximal element of $D$ then it is a decomposition of $M$ into indecomposable summands.

By Zorn's lemma, it suffices to prove that any non-empty chain $T$ in $D$ has an upper bound.
We consider the inverse limit
\[
L = \varprojlim_{I\in T} I
\]
using the maps $f_{IJ}$. 
An element $\lambda\in L$ is given by $\lambda_I \in I$ for all $I \in T$, satisfying $f_{IJ}(\lambda_J) = \lambda_I$ for all $I\le J$ in $T$,
and we define
\[
M[\lambda] = \bigcap_{I\in T} \lambda_I,
\]
a submodule of $M$. 
We show that 
\[
M = \bigoplus_{\lambda\in L} M[\lambda].
\]

Suppose $x$ is an object in $\cC$ and we have a relation 
\[
m_1+\dots+m_n = 0
\]
with $m_i \in M[\lambda^i]_x$ for distinct $\lambda^i \in L$. 
For $i\neq j$ we have $\lambda^i \neq \lambda^j$, so $\lambda^i_I \neq \lambda^j_I$ for some $I$.
But then also $\lambda^i_J \neq \lambda^j_J$ whenever $I\le J$.
Repeating for all pairs $i\neq j$, and using that $T$ is a chain, there is some $J$ with $\lambda^1_J,\dots,\lambda^n_J$ distinct.
But then since $M$ is the direct sum of the elements of $J$, and $m_i \in M[\lambda^i]_x \subseteq (\lambda^i_J)_x$, 
we deduce that $m_i=0$ for all $i$.

Suppose that $m\in M_x$ and $m\neq 0$.
For any $I\in T$ we can write
\[
m = m_1+\dots+m_n
\]
with $n\ge 1$ and the $m_i$ non-zero and belonging to $(N_i)_x$ for distinct elements $N_i$ of $I$.
Moreover 
\[
n \le \dim \bigoplus_{i=1}^n (N_i)_x \le \dim M_x.
\]
Choose $I$ such that the decomposition of $m$ has $n$ maximal.
For any $J$ in $D$ with $I \le J$, the submodule $N_i$ breaks up as a direct sum of elements of $J$, but the element $m_i$
does not become a non-trivial sum of terms. Thus $m_i$ must belong to one of the submodules in $J$.
This defines an element $\lambda^i \in L$, and $m_i \in M[\lambda^i]_x$. Thus $m\in\sum_{\lambda\in L} M[\lambda]_x$.

Thus, as claimed, $M = \bigoplus_{\lambda\in L} M[\lambda]$.
We now delete any terms from the sum which are zero. 
Letting $U = \{ M[\lambda] : \text{$\lambda\in L$ and $M[\lambda]\neq 0$}\}$
we have $M = \bigoplus_{N \in U} N$ and so $U\in D$. Clearly $U$ is an upper bound for $T$, as required.

\section{Decomposition into interval modules}
In this section we prove Theorem~\ref{t:totorder}. Let $M\colon S\to \cvec$ for a totally ordered set $S$. 
%
%
%
The support of an indecomposable persistence module over a totally ordered set must necessarily be an interval. Hence, it suffices to show that if $M$ is indecomposable with support $I$, then $M$ is isomorphic to $k_I$.  Furthermore, we may assume without loss of generality that the support of $M$ is the whole of $S$. 

We show first that if $S$ has a minimal element $s$, then $M$ is isomorphic to $k_S$. 
Since $M_s\neq 0$ we can choose $0\neq m\in M_s$. Let $J=\{x\in S\mid M_{\iota_{xs}}(m)\neq 0\}$ and define a monomorphism $k_J\to M$, by sending the canonical basis element of the vector space $(k_S)_x$ to $M_{\iota_{xs}}(m)$. The constant module $k_J$ is injective by \cref{l:injective}, so the morphism is a split monomorphism. Since $M$ is indecomposable, it must be an isomorphism. We conclude that $M\cong k_J=k_S$.

Next let $M$ be a pointwise finite-dimensional indecomposable persistence module. 
We will show that the map $M_{\iota_{yx}}\colon M_x\to M_y$ is surjective for all $x<y$.
Consider the restriction $M'$ of $M$ to $S' = \{ s\in S : s\ge x\}$.
This is a pointwise finite-dimensional persistence module over $S'$, so it is a direct sum of indecomposables.
Take one of the indecomposable summands $N$ of $M'$.
If $N_x=0$ then the projection and inclusion maps $M'\to N\to M'$ 
extend to give maps $M\to N\to M$, so $N$ is a summand of $M$, a contradiction.
Thus by the remark above, $N_x$ is an interval module.
Thus $M'$ is isomorphic to a direct sum of interval modules for intervals with minimal element $x$.
This shows that the maps $M_x\to M_y$ are surjective for all $x<y$.  The result now follows from \cref{l:directecodirected} and \cref{t:decomp}.


\section{Decomposition of Middle Exact Bi-Modules}
In this section we prove \cref{thm:block}. Let $S$ and $T$ be totally ordered sets and let $P=S\times T$ denote their product.

\begin{dfn}
A persistence module $M\colon P\to \cvec$ is \emph{middle exact} if 
\begin{equation}
0\rightarrow M_a \xrightarrow{M_{\iota_{ba}}\oplus M_{\iota_{ca}}} M_b\oplus M_c \xrightarrow{(M_{\iota_{db}}, -M_{\iota_{dc}})} M_d\rightarrow 0
\label{eq:middlex}
\end{equation}
is middle exact (i.e. exact over the middle term) whenever $a=(x,y)$, $b=(x,y')$, $c=(x',y)$ and $d=(x', y')$. The trivial vector spaces of \cref{eq:E} have been included for convenience as we will also consider the case that \cref{eq:E} is in fact short exact. 
\label{def:middleex}
\end{dfn}
\begin{dfn} A non-empty subset $I\subseteq P$ is a \emph{block} if:
\begin{enumerate}
\item $I=J_S\times J_T$ for interval ideals $J_S$ and $J_T$, 
\item $I=J_S\times J_T$ for interval filters $J_S$ and $J_T$,
\item $I=J_S \times T$ for an interval $J_S$, 
\item $I=S\times J_T$ for an interval $J_T$. 
\end{enumerate}
We shall refer to these as blocks of type death (\sdb), birth (\sbb), vertical ({\svb}) and horizontal (\shb), respectively. Observe that one block may be of several types.
\end{dfn}
We say that $k_I$ is a \emph{block module} whenever $I$ is a block. Observe that if $I$ is of type {\sdb}, then $I$ is a directed ideal. Hence, $k_I$ is injective by \cref{l:injective}.
For $x\in S$ and $y\in T$ define subposets
\begin{align*}
(x,y)^{\overrightarrow{\leftarrow}}&:= \left(\{x\}\times (-\infty, y]\right)\cup \left((-\infty, x]\times \{y\}\right)\subseteq S\times T\\
(x,y)^{\overleftarrow{\rightarrow}}&:= \left(\{x\}\times [y, \infty)\right)\cup \left([x, \infty)\times \{y\}\right)\subseteq S\times T.\end{align*}
Recall that a non-empty subset $I$ of $(x,y)^{\overrightarrow{\leftarrow}}$ or $(x,y)^{\overleftarrow{\rightarrow}}$ is an interval if it is convex and connected. 
\begin{lem}
Let $M\colon (x,y)^\star \to \cvec$ be pointwise finite-dimensional and indecomposable for $\star\in \{\rightleftarrows, \leftrightarrows\}$. Then $M\cong k_I$ for some interval $I$. 
\label{thm:zz}
\end{lem}
\begin{proof}
The two cases are dual so it suffices to prove it for the case $\star =~\rightleftarrows$. Let $M^\ell$ denote the restriction of $M$ to $(-\infty, x]\times \{y\}$. Assume that $\ker M_\alpha \neq 0$ for some $\alpha\colon (t,y) \to (x,y)$. Then $\ker M^\ell_\alpha \neq 0$, and by \cref{t:totorder}, $M^\ell$ has a summand $k_I$, where $I\subseteq (-\infty, x)\times \{y\}$ is an interval. Since $(x,y)\not\in I$, this shows that $k_I$ extends to a summand of $M$ and thus $M\cong k_I$. The corresponding argument applies if $\ker M_\alpha \neq 0$ for some $\alpha\colon (x,t) \to (x,y)$. To conclude the proof it suffices to consider the case that $M_\alpha$ is injective for all $\alpha\colon p\to (x,y)$. As $\dim M_{(x,y)} < \infty$, we can choose indices 
 \begin{align*}
  -\infty &= a'_{0} < a'_{1} < \cdots < a'_{{n-1}} < a'_{n}= y\\
   -\infty &= a_{0} < a_{1} < \cdots < a_{{n-1}} < a_{n}= x
\end{align*}
such that $M_{(x,t)} \to M_{(x,t')}$ and $M_{(s,y)} \to M_{(s',y)}$ are isomorphisms whenever $t, t'\in (a'_{i}, a'_{{i+1}})$ and $s,s'\in (a_{i}, a_{i+1})$. Thus, by choosing $b_i\in (a_i, a_{i+1})$ and $b_i'\in (a'_i, a'_{i+1})$, we get that $M$ is completely described by the following persistence module
\begin{center}
\begin{tikzpicture}[scale=0.5][baseline= (a).base]
\node[scale=0.8] (a) at (0,0){
 \begin{tikzcd}
M_{(b_0,y)}\ar[r]& M_{(a_1,y)}\ar[r]  & M_{(b_1,y)}\ar[r]& \cdots \ar[r]& M_{(b_{n-1},y)}\ar[r]&M_{(x,y)}\\
M_{(x,b'_0)}\ar[r] & M_{(x,a'_1)}\ar[r]  & M_{(x,b'_1)}\ar[r]& \cdots \ar[r]& M_{(x,a'_{n-1})}\ar[r]& M_{(x, b'_{n-1})}\ar[u]
\end{tikzcd}};
\end{tikzpicture}
\end{center}
A decomposition of this persistence module lifts to a decomposition of $M$. It follows from the representation theory of the linear quiver $A_n$ that $M\cong k_I$ for some interval $I$, see for example \cite[Theorem 1.1]{ringdynkin}.
\end{proof}

For $(s,t)\in S\times T$, let ${\bf v}_s = \{ (s,y) \mid y\in T\}$, ${\bf h}_t = \{ (x,t)\mid x\in S\}$, and let $M^{\bf v_s}$ and $M^{\bf h_t}$ denote the respective restrictions of $M$ to ${\bf v}_s$ and ${\bf h}_t$. 
\begin{lem}
Assume that $M$ is pointwise finite-dimensional and middle exact. Let $s\in S, t\in T$, and let $J_S\subseteq S$ and $J_T\subseteq T$ be intervals. 
\begin{enumerate}
\item Assume that there exists an upper bound for $J_T$ in $T-J_T$. A monomorphism $h\colon k_{\{s\}\times J_T} \hookrightarrow M^{\bf v_s}$ lifts to a monomorphism \[h\colon k_{(-\infty, s]\times J_T} \hookrightarrow M|_{(-\infty, s]\times T}.\]
\item Assume that there exists an upper bound for $J_S$ in $S-J_S$.  A monomorphism $h\colon k_{J_S\times \{t\}} \hookrightarrow M^{\bf h_t}$ lifts to a monomorphism \[h\colon k_{J_S\times (-\infty, t]} \hookrightarrow M|_{S\times (-\infty, t]}.\]
\end{enumerate}
\label{lem:lift}
\end{lem}
\begin{proof}
We prove the first case; the second case is symmetrical. For $p=(p_1, p_2)\in (-\infty,s]\times J_T$ let $\pi_{J}(p)\colon p\to (s,p_2)$. Write $\epsilon>J_T$ if $\epsilon\in T-J_T$ and $\epsilon$ is an upper bound for $J_T$. For $\epsilon > J_T$ define 
\[\alpha_{p^\epsilon}\colon p \to (p_1, \epsilon)\]
\[E_p^\epsilon = M_{\pi_{J}(p)}^{-1}(\Ima h_{(s,p_2)}) \bigcap \ker M_{\alpha_{p^\epsilon}}.\]
It follows from the middle exactness condition on $M$ that $E_p^\epsilon \neq 0$, and that $E_q^\epsilon \to E_p^\epsilon$ is a surjection for all $q\leq p$ in $(-\infty,s]\times J_T$. Now consider 
\[E_p := \bigcap_{\epsilon>J_T} E_p^\epsilon.\] Since $M$ is pointwise finite-dimensional, there exists an $\epsilon_p>J_T$ such that $E_p = E^{\epsilon_p}_p$, and therefore it is also true that $E_p\neq 0$, and that the map $E_q\rightarrow E_p$ is a surjection for all $q\leq p$. 
%
%
Since $(-\infty, s]\times J_T$ is a product of totally ordered sets, it is both directed and codirected.   Hence it follows from \cref{l:directecodirected} that $E\colon (-\infty, s]\times J_T\to \cvec$ has a copy of $k_{(-\infty, s]\times J_T}$ as a summand. The multiple of the canonical inclusion $k_{(-\infty, s]\times J_T} \hookrightarrow E$ which agrees with $h$ on $\{ s \}\times J_T$  defines a lift of $h$.\end{proof}
\begin{lem}
Let $M$ be pointwise finite-dimensional, middle exact and indecomposable. If there exist $a,b,c,d$ as in \cref{def:middleex} such that $\ker M_{\iota_{ba}}\cap \ker M_{\iota_{ca}} \neq 0$, then $M\cong k_I$ where $I$ of type \sdb. 
\label{lem:splitinj}
\end{lem}
\begin{proof}
By assumption, the restriction of $M$ to $(x,y)^{\leftrightarrows}$ must contain a summand isomorphic to $k_J$, where $J=\left( \{x\}\times J_T\right) \cup \left( J_S \times \{y\}\right)$ and $J_S$ and $J_T$ are intervals satisfying:
\begin{itemize}
\item $x\in J_S$ is minimal and $J_S\subseteq [x, x')$,
\item $y\in J_T$ is minimal and $J_T\subseteq [y, y')$. 
\end{itemize}
We shall construct a monomorphism $k_I\hookrightarrow M$ where \[I=\left((-\infty, x]\cup J_S\right) \times \left((-\infty,y]\times  J_T\right).\] Since $k_I$ is injective, it follows that $M\cong k_I$. 

Consider the following subsets of $P$:
\begin{align*}
I_1=J_S\times J_T \qquad I_2 = (-\infty, x] \times J_T \qquad I_3 = \left((-\infty, x]\cup J_S\right) \times (-\infty, y].
\end{align*}
Observe that $I = I_1\cup I_2\cup I_3$. The proof proceeds in three steps. 

\textbf{Step 1: Constructing $k_{I_1}\hookrightarrow M$}. Let $N \subseteq M|_{(x,y)^{\leftrightarrows}}$ be such that $M|_{(x,y)^{\leftrightarrows}}=N\oplus N^\bullet$ and $N\cong k_{J}$, and choose $0\neq m \in N_{(x,y)}\subseteq M_{(x,y)}$. 
We shall show that $M_\alpha(m)\neq 0$ for all $\alpha\colon (x,y)\to p$ where $p\in J_S\times J_T$.  Assume for the sake of contradiction that $M_\alpha(m) = 0$ for $\alpha\colon (x,y)\to p=(p_1, p_2)$. By the middle exact sequence
\[M_{(x,y)} \to M_{(x,p_2)}\oplus M_{(p_1, y)} \to M_{(p_1, p_2)}\]
there exists an element $\hat{m} \in M_{(x,y)}$ such that $M_{\alpha'}(\hat{m}) = M_{\alpha'}(m)$ and $M_{\alpha''}(\hat{m}) = 0 $, for $(p_1,y) \xleftarrow{\alpha'} (x,y) \xrightarrow{\alpha''} (x,p_2)$. The first equality, together with the direct sum decomposition of $M|_{(x,y)^{\leftrightarrows}}$ and the injectivity of $N_{\alpha'}$, give $\hat{m} = m + n^\bullet$ for an $n^\bullet \in N^\bullet_{(x,y)}$. Substituting this into the second equality yields $M_{\alpha''}(m) = -M_{\alpha''}(n^\bullet)$. Since $M_{\alpha''}(m)\neq 0$, this contradicts  $M|_{(x,y)}^{\leftrightarrows} = N\oplus N^\bullet$. For any $\alpha\colon (x,y)\to (p_1, p_2)\not\in I_1$, it follows by commutativity that $M_\alpha(m) = 0$. Hence, we have a well-defined monomorphism $h\colon k_{I_1}\hookrightarrow M$ given by $h_{p}(1) = M_{\alpha}(m)$ for $\alpha\colon (x,y)\to p$. 

\textbf{Step 2: Constructing $k_{I_1\cup I_2}\hookrightarrow M$.} The $h$ of the previous step restricts to a monomorphism $h'\colon k_{\{x\}\times J_T}\hookrightarrow M^{\bf v_x}$. By (1) of \cref{lem:lift} this restriction extends to a monomorphism $h'\colon k_{I_2}\hookrightarrow M|_{(-\infty, x]\times T}$. This defines a lift of $h$ to $h\colon k_{I_1\cup I_2}\hookrightarrow M$. 

\textbf{Step 3: Constructing $k_{I_1\cup I_2\cup I_3}\hookrightarrow M$.} The $h$ of Step 2 restricts to a monomorphism $h''\colon k_{\left((-\infty, x]\cup J_S\right)\times \{y\}} \hookrightarrow M^{\bf h_y}$. By (2) of \cref{lem:lift} this restriction extends to a monomorphism $h''\colon k_{I_3}\hookrightarrow M|_{S \times (-\infty, y]}$.  This defines a lift of $h$ to $h\colon k_{I_1\cup I_2 \cup I_3}\hookrightarrow M$. 
\end{proof}
We also have the following dual lemma.
\begin{lem}
\label{lem:splitproj}
Let $M$ be pointwise finite-dimensional, middle exact and indecomposable. If there exist $a,b,c,d$ as in \cref{def:middleex} such that $\Coker ((M_{\iota_{db}}, -M_{\iota_{dc}}))\neq 0$, then $M\cong k_I$ where $I$ of type \sbb. 
\end{lem}
\begin{proof}
Observe that $DM$ is middle exact whenever $M$ is, and that $I$ is a directed ideal in $(S\times T)^{\rm op}$. Since $M\cong D^2M$ we also have that $DM$ is indecomposable. In particular, $DM \cong Dk_I$ by \cref{lem:splitinj}, and thus $k_I \cong D^2(k_I) \cong D^2M \cong M$. 
\end{proof}
The previous two lemmas show that it suffices to consider the case where \cref{eq:middlex} is \emph{short exact}. 
Define persistence modules \[\Ima M^\leftarrow, \Ima M^\downarrow, \ker M^\rightarrow, \ker M^\uparrow\colon P \to \cvec\] by
\begin{align*}
\Ima M^\leftarrow_{(p_1, p_2)} &= \bigcap_{\alpha: (q, p_2)\to (p_1, p_2)} \Ima M_\alpha, \quad \ker M^\rightarrow_{(p_1, p_2)} = \bigcup_{\alpha: (p_1, p_2)\to (q, p_2)} \ker M_\alpha \\
\Ima M^\downarrow_{(p_1, p_2)} &= \bigcap_{\alpha: (p_1, q)\to (p_1, p_2)} \Ima M_\alpha, \quad \ker M^\uparrow_{(p_1, p_2)} = \bigcup_{\alpha: (p_1, p_2)\to (p_1, q)} \ker M_\alpha
\end{align*}
It is not hard to see that these are submodules of $M$. By definition, $M_{\alpha}$ maps $\Ima M^\leftarrow_{(p_1, p_2)}$ onto $\Ima M^\leftarrow_{(q, p_2)}$ for any $\alpha\colon (p_1, p_2)\to (q, p_2)$. Let $\alpha\colon (p_1, p_2)\to (p_1,q)$. Since $M$ is pointwise finite-dimensional, there exists $s\in S$ such that $\Ima M^\leftarrow_{(p_1, p_2)} = \Ima M_\beta$ and $\Ima M^\leftarrow_{(p_1, q)} = \Ima M_{\beta'}$ where $\beta\colon (s, p_2)\to (p_1, p_2)$ and $\beta'\colon (s, q)\to (p_1,q)$. This shows that $\Ima M^\leftarrow$ is a submodule of $M$. The other cases are similar.

Following the same line of arguments we also have the following simple lemma.
\begin{lem}
Let $M$ be pointwise finite-dimensional, middle exact and assume that \cref{eq:middlex} is short exact for all $a,b,c,d$ as in \cref{def:middleex}. Then $\ker M^\rightarrow \cap \ker M^\uparrow =0$ and $M = \Ima M^\leftarrow+ \Ima M^\downarrow$.
\label{lem:shortexact}
\end{lem}
\begin{lem}
Let $M$ be as in \cref{lem:shortexact}. If $\Ima M^\leftarrow\cap \ker M^\rightarrow \neq 0$ or $\Ima M^\downarrow\cap \ker M^\uparrow\neq 0$, then $M\cong k_I$ where $I$ is of type \sdb.
\label{lem:kerker}
\end{lem}
\begin{proof}
We prove it for the first case; the second case is symmetrical. Let $W=\Ima M^\leftarrow\cap \ker M^\rightarrow$ and assume that $W_{(x, y)} \neq 0$. By \cref{t:totorder} and the assumptions on $W$, the restriction $W^{\bf h_{y}}$ decomposes as a direct sum $\oplus_J k_J$ where at least one interval ideal $J$ has an upper bound in ${\bf h_{y}}-J$.  Fix such $J$ and consider the associated monomorphism $h\colon k_{J\times \{y\}} \hookrightarrow W^{\bf h_{y}} \subseteq M^{\bf h_{y}}$. By \cref{lem:shortexact}, $\ker M^\rightarrow_{(s,y)}\cap \ker M^\uparrow_{(s,y)} = 0$, and therefore we must have $M_\alpha(h_{(s,y)}(1)) \neq 0$ for all $\alpha\colon (s,y)\to (s, p_2)$. Hence, $h$ lifts to a monomorphism $k_{J\times [y, \infty)}\hookrightarrow M$. This monomorphism can in turn be lifted to $h: k_{J\times T} \to M$ by means of (2) of \cref{lem:lift}. Since $J\times T$ is of type {\sdb} the result follows. 
\end{proof}

We are now ready to prove the main statement of this section.
\begin{proof}[Proof of \cref{thm:block}]
By \cref{t:totorder} it suffices to show the result for $M$ indecomposable. If the conditions of \cref{lem:splitinj} or \cref{lem:splitproj} are satisfied, then we are done. Thus, we may assume that \cref{eq:middlex} is short exact. Consider the submodules $\Ima M^\leftarrow$ and $\Ima M^\downarrow$, and an arbitrary $(x,y)\in P$. By \cref{lem:kerker} we may assume that $\ker (\Ima M^\leftarrow_\alpha) = 0$ and $\ker (\Ima M^\downarrow_\beta) =0$ for all $\alpha\colon(x,y) \to (x',y)$ and $\beta\colon (x,y)\to (x, y')$. Since these morphisms are surjective by definition, it follows that they are in fact isomorphisms. Hence, if $(\Ima M^\leftarrow)^{\bf v_x} \cong \bigoplus_J k_J$, then $\Ima M^\leftarrow \cong \bigoplus_J k_{S\times J}$, and therefore block-decomposable. Symmetrically we also get that $\Ima M^\downarrow$ is block-decomposable. By \cref{lem:shortexact} we have that $M=\Ima M^\leftarrow + \Ima M^\downarrow$. Let $W=\Ima M^\leftarrow\cap\Ima M^\downarrow$ and observe that the internal morphisms of $W$ are all isomorphisms. Thus, if $W\neq 0$, then we have a monomorphism $k_{P}\hookrightarrow W \subseteq M$, and therefore $M\cong k_{P}$. If $X=0$, then $M=\Ima M^\leftarrow\oplus \Ima M^\downarrow$, and since $M$ is indecomposable, $M=\Ima M^\leftarrow$ or $M=\Ima M^\downarrow$. 
\end{proof}

\subsection{Decomposition of Infinite Zigzags}
\label{sec:zigzag}
Define a \emph{zigzag path} $\gamma$ to be a function $\gamma: \Z \to \Z^2$ satisfying
\[\gamma(i+1) \in \left\{\gamma(i) + (1,0), \gamma(i) - (0,1)\right\}\]
and $\lim_{i\to \pm \infty} \gamma(i) = (\pm \infty, \mp \infty)$.  For such a path $\gamma$ let $Z(\gamma)\subseteq \R^2$ be the poset 
\[Z(\gamma) := \{ (s,t)\in \R^2 \mid \exists i\in \Z \text{ such that } \gamma(i) \leq (s,t) \leq \gamma(i+1)\}. \]
Observe that $Z(\gamma)$ separates $\R^2-Z(\gamma)$ into two disjoint subsets 
\begin{align*}
R_U&=\{(s,t) \mid \exists p\in Z(\gamma) \text{ such that } (s,t)\geq p\}-Z(\gamma)\\
R_L&= \{(s,t) \mid \exists p\in Z(\gamma) \text{ such that } (s,t) \leq p\}-Z(\gamma).
\end{align*}
We say that a non-empty subset $I\subseteq Z(\gamma)$ is an \emph{interval} if it is convex and connected. Observe that a non-trivial intersection of a block and $Z(\gamma)$ is an interval.

\begin{cor}
Let $\gamma$ be a zigzag path. If $M\colon Z(\gamma)\to \cvec$ is pointwise finite-dimensional, then $M$ decomposes into interval modules. 
\label{thm:zigzag}
\end{cor}
To prove this we need the following lemma
\begin{lem}
Let $M\colon \R^2\to \cvec$ be such that $M|_{[i, i+1]\times [j, j+1]}$ is middle exact for all $(i,j)\in \Z^2$. Then $M$ is middle exact.
\label{lem:subdivide}
\end{lem}
\begin{proof}
Let $a,b,c,d$ as in \cref{def:middleex} and choose any point $a\leq (s, t)\leq d$. Consider the following commutative diagram
\[
\begin{tikzcd}
M_b\ar[r] & M_{(s, y')} \ar[r] & M_d\\
M_{(x, t)} \ar[r]\ar[u] & M_{(s,t)} \ar[r]\ar[u] & M_{(x', t)}\ar[u]\\
M_a \ar[u]\ar[r] & M_{(s,y)} \ar[u]\ar[r] & M_{c}\ar[u]
\end{tikzcd}
\]
A simple diagram chase shows that if $M$ satisfies the middle exact condition on the four minimal rectangles, then so does it on the larger bounding rectangle. Thus, we may iteratively subdivide the bounding rectangle such that the corner points of any (non-trivial) minimal rectangle  all lie in a square $[i, i+1]\times [j, j+1]$ for some $(i,j)$. 
\end{proof}
Let $\lceil t\rceil$ denote the least  integer \emph{strictly} greater than $t$, and let $\lfloor t\rfloor$ denote the greatest integer \emph{strictly} less than $t$. 

We can extend $M$ to a representation $E(M)\colon \R^2\to \cvec$ recursively as follows
\begin{equation}
E_\gamma(M)_{(s,t)} =
\begin{cases} 
M_{(s,t)} &\text{ if } (s,t)\in Z(\gamma)\\
\Ker \left( M_{(s,\lceil t\rceil)}\oplus M_{(\lceil s \rceil ,t)}\to M_{(\lceil s\rceil ,\lceil t \rceil )}\right) &\text{ if } (s,t)\in R_L\\
\Coker \left(M_{(\lfloor s \rfloor, \lfloor t\rfloor)} \to M_{(s,\lfloor t\rfloor )}\oplus M_{(\lfloor s \rfloor ,t)}\right) &\text{ if } (s,t)\in R_U
\end{cases}
\label{eq:E}
\end{equation}
where the internal morphisms are given by functoriality of $\Ker$ and $\Coker$. This definition is well-defined as every recursive call will terminate in finite time. An equivalent definition of $E_\gamma(M)$ using limits and colimits can be gives as follows: for $(s,t)\in \R^2$ let $D(s,t) = \{p\in \R^2\mid p \leq (s,t)\}$ and $U(s,t) = \{p\in \R^2 \mid p\geq (s,t)\}$. Then $E_\gamma(M)$ is the following persistence module
\[
E_\gamma(M)_{(s,t)} =
\begin{cases} 
M_{(s,t)} &\text{ if } (s,t)\in Z(\gamma)\\
\varprojlim M|_{Z(\gamma)\cap U(s,t)} &\text{ if } (s,t)\in R_L\\
\varinjlim M|_{Z(\gamma)\cap D(s,t)} &\text{ if } (s,t)\in R_U
\end{cases}.
\]
By \cref{eq:E} we see that $E_\gamma(M)$ is middle exact on every square $[i, i+1]\times [j,j+1]$ and thus middle exact by \cref{lem:subdivide}. As $E_\gamma(M)$ is clearly pointwise finite-dimensional it follows from \cref{thm:block} that $E_\gamma(M)$ is block-decomposable. Therefore \[M=E_\gamma(M)|_{Z(\gamma)} \cong (\oplus_J k_J)|_{Z(\gamma)}\cong \oplus_J k_{J\cap Z(\gamma)}\] where the $J$'s are blocks. This concludes the proof of \cref{thm:zigzag}. 

\subsection{Upper-triangular support}
\label{sec:upper}
In this section $T\subseteq \R^2$ denotes the \emph{strictly upper-triangular} subset $\{(x, y)\in \R^2 \mid x+y > 0\}$, and $\overline T\subseteq \R^2$ denotes the \emph{upper-triangular} subset $\{(x, y)\in \R^2 \mid x+y \geq 0\}$. We define a block in $T$ to be a subset of the form $J\cap T$, where $J\subseteq \R^2$ is a block. Furthermore, $M\colon T\to \cvec$ is \emph{middle exact} if \cref{eq:middlex} is middle exact for all such $a,b,c,d\in T$. Blocks and middle exact modules are defined accordingly in the upper-triangular setting. 

First we prove \cref{thm:upperT} in the strictly upper-triangular setting. Observe that if $I\subseteq \R^2$ is of type {\sdb}, then $I\cap T$ is both an ideal and directed. Hence, $k_{I\cap T}\colon T\to \cvec$ is injective by \cref{l:injective}.
\begin{lem}
Let $M\colon T\to \cvec$ be pointwise finite-dimensional, middle exact and indecomposable. If there exist $a,b,c,d\in T$ as in \cref{def:middleex} such that \[\ker (M_{\iota_{ba}}) \cap \ker (M_{\iota_{ca}}) \neq 0,\] then $M\cong k_{I\cap T}$ where $I$ of type \sdb. 
\label{lem:splitinjT}
\end{lem}
\begin{proof}
The restriction of $M$ to $((x-y)/2, \infty)\times ((y-x)/2, \infty)$ is
again middle exact and by \cref{lem:splitinj} it must have a summand
isomorphic to $k_{R_0}$ where $R_0=((x-y)/2, s''\rangle \times ((y-x)/2,
t''\rangle$ for $s'',t''\in \R$.
This defines a monomorphism $f_0\colon k_{R_0}\hookrightarrow M$ of
persistence modules for $T$.
Let $I = (-\infty, s''\rangle\times (-\infty, t''\rangle$, and write $J =
I\cap T$ as a disjoint union $\bigcup_{n=0}^\infty R_n$, where

(i) each $R_n$ is of the form $(x_n,x_n'\rangle\times (y_n,y_n'\rangle$, and

(ii) $J\setminus J_n$ is an ideal in $T$ for all $n$, where $J_n =
\bigcup_{i=0}^n R_n$.
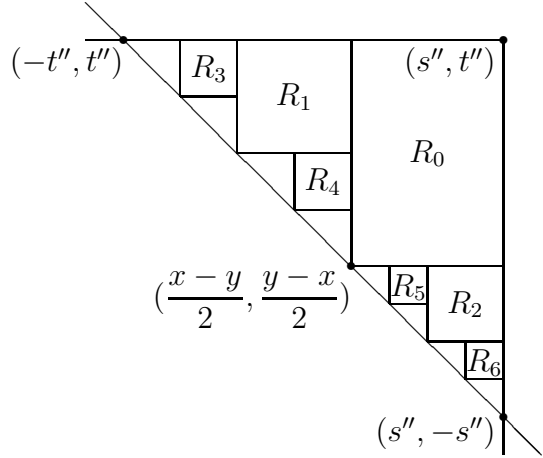
\[
\setlength{\unitlength}{0.5cm}
\begin{picture}(12,12)(-6,-6)
\put(-6,6){\line(1,-1){12}}
\put(5,5){\line(-1,0){11}}
\put(5,5){\line(0,-1){11}}
\put(1,-1){\line(1,0){4}}
\put(1,-1){\line(0,1){6}}
\put(3,-3){\line(1,0){2}}
\put(3,-3){\line(0,1){2}}
\put(-2,2){\line(1,0){3}}
\put(-2,2){\line(0,1){3}}
\put(-3.5,3.5){\line(1,0){1.5}}
\put(-3.5,3.5){\line(0,1){1.5}}
\put(-0.5,0.5){\line(1,0){1.5}}
\put(-0.5,0.5){\line(0,1){1.5}}
\put(2,-2){\line(1,0){1}}
\put(2,-2){\line(0,1){1}}
\put(4,-4){\line(1,0){1}}
\put(4,-4){\line(0,1){1}}
\put(3,2){\makebox(0,0){$R_0$}}
\put(-0.5,3.5){\makebox(0,0){$R_1$}}
\put(4,-2){\makebox(0,0){$R_2$}}
\put(-2.75,4.25){\makebox(0,0){$R_3$}}
\put(0.25,1.25){\makebox(0,0){$R_4$}}
\put(2.5,-1.55){\makebox(0,0){$R_5$}}
\put(4.5,-3.55){\makebox(0,0){$R_6$}}
\put(-5,5){\circle*{0.2}}
\put(-5,4.8){\makebox(0,0)[tr]{$(-t'',t'')$}}
\put(5,-5){\circle*{0.2}}
\put(4.8,-5){\makebox(0,0)[tr]{$(s'',-s'')$}}
\put(5,5){\circle*{0.2}}
\put(4.8,4.8){\makebox(0,0)[tr]{$(s'',t'')$}}
\put(1,-1){\circle*{0.2}}
\put(1,-1){\makebox(0,0)[tr]{$\displaystyle (\frac{x-y}{2},\frac{y-x}{2})$}}
\end{picture}
\]
By induction we extend $f_0$ to a monomorphism
$f_n\colon k_{J_n}\hookrightarrow M$ for all~$n$.
Namely, suppose we are given $f_{n-1}$, we construct $f_n$.
There are two situations we need to consider:

(a) where points above and to the right of $R_n$ are in $J_{n-1}$ (for
example $R_4$, $R_5$), and

(b) where points to the right of $R_n$ are in $J_{n-1}$ and points above
$R_n$ are not in $J$ (for example $R_1$, $R_3$), or the dual
situation (for example $R_2$, $R_6$).

For $p\in R_n$ we construct a set $\emptyset \neq E_p \subseteq M_p$ as
follows.
For situation (a), let $q\in J_{n-1}$ be a point above $p$ and let $s\in
J_{n-1}$ be a point to the right of $p$.
We complete them to a rectangle $pqrs$. Then $r\in J_{n-1}$, and
$(f_{n-1})_q(1) \in M_q$ and $(f_{n-1})_s(1)\in M_s$ have the same image
$(f_{n-1})_r(1)\in M_r$.
By middle exactness, the set
\[
E_p = \{ m\in M_p : \text{$M_{\iota_{qp}}(m) = (f_{n-1})_q(1)$ and
$M_{\iota_{rp}}(m) = (f_{n-1})_r(1)$} \}
\]
is not empty.
For situation (b), let $q\notin J$ be a point above $p$ and let $s\in
J_{n-1}$ be a point to the right of $p$.
We complete them to a rectangle $pqrs$. Then $r\notin J$ and
$0 \in M_q$ and $(f_{n-1})_s(1)\in M_s$ have the same image $0\in M_r$.
By middle exactness, the set
\[
E_p = \{ m\in M_p : \text{$M_{\iota_{qp}}(m) = 0$ and $M_{\iota_{rp}}(m)
= (f_{n-1})_r(1)$} \}
\]
is not empty.

For a different choice of $q',s'$ with $q'<q$ and $s'<s$ in both cases
(a) and (b) we obtain a set $E_p' \subseteq E_p$. But the set $E_p$ is a
coset of $\Ker M_{\iota_{qp}}\cap \Ker M_{\iota_{sp}}$. Henceforth, in
the definition of $E_p$, we choose $q$ and $s$ such that this subspace
is of minimal dimension.
Thus $E_p' = E_p$ for any choice of $q',s'$ as above.
It follows that for $m\in E_p$ and $t\in J_{n-1}$ with $p<t$, we have
$M_{\iota_{tp}}(m) = (f_{n-1})_t(1)$,
and for $t\notin J$ with $p<t$ we have $M_{\iota_{tp}}(m) = 0$.

Now if $p,p'\in R_n$ and $p'\le p$ then middle exactness ensures that
the map $E_{p'}\to E_p$ is surjective.
To see this we can reduce to the cases when $p'$ is to the left of, or
below $p$. We deal with the first of these.
We choose a rectangle $p'pqq'$ where $q'$ is above $p'$ and $q$ is above
$p$, both valid for the definition of $E_p$ and $E_{p'}$.
The vertical condition for $m\in E_p$ is that $M_{\iota_{qp}}(m)$ is
equal to $(f_{n-1})_q(1)$ in case (a) and 0 in case (b).
The vertical condition for $m'\in E_{p'}$ is that $M_{\iota_{q'p'}}(m')$
is equal to $(f_{n-1})_{q'}(1)$ in case (a) and 0 in case (b).
Middle exactness for the rectangle $p'pqq'$ thus implies that $E_{p'}\to
E_p$ is surjective.

Choose a sequence $p_1 \ge p_2 \ge \dots$ of elements of $R_n$, such
that for any $p\in R_n$, $p_i\le p$ for some $i$. By recursively lifting elements we get that
\[
V = \lim_{\substack{\longleftarrow \\ i}} E_{p_i}
\]
is non-empty. Choose $v\in V$ and define $f\colon k_{R_n}
\hookrightarrow M|_{R_n}$ by
\[
f_p(1)= \text{$M_{\iota_{p,p_i}}(v_i)$ where $p_i\leq p$.}
\]
This defines a lift of $f_{n-1}$ to a monomorphism $f_n\colon k_{J_n}
\hookrightarrow M$, as required.

Combining these maps gives a monomorphism $f\colon k_J\hookrightarrow M$.
Since $k_J$ is injective  and $M$ is indecomposable, we deduce that
$M\cong k_J$, as required.
\end{proof}

We also have the following result which is a direct consequence of \cref{lem:splitproj}.
\begin{lem}
\label{lem:splitprojT}
Let $M\colon T\to \cvec$ be pointwise finite-dimensional, middle exact and indecomposable. If there exist $a,b,c,d\in T$ as in \cref{def:middleex} such that \[\Coker ((M_{\iota_{db}}, -M_{\iota_{dc}}))\neq 0\] then $M\cong k_{I\cap T}$ where $I$ of type \sbb. 
\end{lem}
\begin{proof}The restriction $M'$ of $M$ to $U(a) = \{p\mid p\geq a\}$ is again middle exact, and by \cref{lem:splitproj} it has a summand isomorphic to $k_I$ where $I$ is a block of type $\sbb$ contained in the interior of $U(a)$. Since $I$ is contained in the interior of $U(a)$ it follows that the inclusion and projection $k_I\hookrightarrow M'\twoheadrightarrow k_I$ extend to give maps $k_I\hookrightarrow M\twoheadrightarrow k_I$. This shows that $M\cong k_I = k_{I\cap T}$. 
\end{proof}

\begin{proof}[Proof of \cref{thm:upperT} (Strictly upper-triangular) ]
By \cref{t:decomp} it suffices to consider the case that $M$ is indecomposable. Furthermore, \cref{lem:splitinjT,lem:splitprojT} allow us to restrict our attention to the case that \cref{eq:middlex} is short exact for all such $a,b,c,d\in T$. In particular, this means that we have the following natural isomorphisms
for all such $a,b,c$ and $d$:
\begin{align}
M_d \cong \Coker(M_a \to M_b\oplus M_c) & & M_a \cong \Ker(M_b\oplus M_c \to M_d).
\label{eq:T}
\end{align}
Consider any zigzag path $\gamma$ satisfying $\Ima \gamma \subset T$. By comparing \cref{eq:T} to \cref{eq:E} we see that $M \cong E_\gamma(M|_{Z(\gamma)})|_T$, and by \cref{thm:zigzag}, \[\ E_\gamma(M|_{Z(\gamma)})|_T\cong E_\gamma\left(\bigoplus_{I} k_I\right)\Big |_T \cong \bigoplus_I E_\gamma(k_I)|_T.\]
Since $M$ is assumed to be indecomposable it follows that $M\cong E_\gamma(k_I)|_T$ where $I= J\cap Z(\gamma)$ for a block $J\subseteq \R^2$. It is straightforward to verify that $E_\gamma(k_{J\cap Z(\gamma)})|_T = k_{J\cap T}$ if $J\cap Z(\gamma)\neq \emptyset$. 
\end{proof}
\begin{proof}[Proof of \cref{thm:upperT} (Upper-triangular)]
We shall show that any indecomposable, pointwise finite-dimensional and middle exact persistence module $N\colon \overline T\to \cvec$ is a block module.  This will be done by first restricting $N$ to $T$,  and then extending $N|_T$ to a module over $\overline T$. We show that when the restriction $N|_T$ is non-zero, the composition given by first restricting and then extending is an isomorphism. The result now follows from our previous work in the strictly upper-triangular setting.

Let $p=(x_0,y_0) \in \overline T \setminus T$, so $x_0+y_0=0$. Given a point $s = (x,y)\in T$ with $x_0<x$ and $y_0<y$, 
we consider the non-degenerate rectangle $pqrs$ where $q=(x,y_0)$ and $r=(x_0,y)$. Observe that $q,r\in T$.
For a persistence module $M\colon T\to \cvec$, we define 
\[
M_p^s = \Ker \left(M_q \oplus M_r \xrightarrow{ (-M_{\iota_{sq}} \ M_{\iota_{sr}}) } M_s \right).
\]
If $s'=(x',y')\in T$ with $x_0<x'\le x$ and $y_0<y'\le y$, and $pq'r's'$ is the corresponding rectangle,
then the commutative diagram
\[
\begin{CD}
M_{q'} \oplus M_{r'} @>(-M_{\iota_{s'q'}}\ M_{\iota_{s'r'}})>> M_{s'} \\
@V \theta VV @V M_{\iota_{ss'}} VV \\
M_q \oplus M_r @>(-M_{\iota_{sq}}\ M_{\iota_{sr}})>> M_s 
\end{CD},
\]
where $\theta = (\begin{smallmatrix} M_{\iota_{qq'}} & 0 \\ 0 & M_{\iota_{rr'}} \end{smallmatrix})$, induces a map $M^{ss'}_p\colon M^{s'}_p \to M^s_p.$

We define $\overline M_p = \varprojlim M^s_p$, where the inverse limit is over all $s$ giving non-degenerate rectangles, as above. Clearly there is a natural map $\overline M_p\to M_s$ for any $s=(x,y)$ with $x_0<x$ and $y_0<y$.
In addition there are natural maps $\overline M_p\to M_q$ for $q = (x,y_0)$ with $x_0<x$ 
and $\overline M_p\to M_r$ for $r = (x_0,y)$ with $y_0<y$.
Using this to extend $M$, this defines a functor from persistence modules over $T$ to persistence modules $\overline M$ over $\overline T$. Observe that if $M \cong k_{J\cap T}$ for a block $J\subseteq \R^2$, then $\overline M \cong k_{J\cap \overline T}$. Moreover, the functor respects direct sum decompositions. 

Now suppose that $N\colon \overline T\to \cvec$. For $p\in \overline T\setminus T$
there is a natural map $\mu_p^s:N_p \to (N|_T)_p^s$ induced by the maps $N_{\iota_{qp}}$ and $N_{\iota_{rp}}$.
This induces a natural map $N_p \to (\overline{N|_T})_p$ and thus a morphism $N\to \overline{N|_T}$.
If $N$ is a middle exact $\overline T$-module, then the map $\mu_p^s$ is surjective.

Now suppose that $N$ is indecomposable, pointwise finite-dimensional, middle exact, and
not isomorphic to the interval module $k_{\{p\}}$ for any point $p\in\overline T\setminus T$.
Since $k_{\{p\}}$ is injective, it follows that it does not occur as a submodule of $N$.

We need to show for any point $p\in \overline T\setminus T$ that $N_p \to (\overline{N|_T})_p$ is an isomorphism.
This map is induced by the maps $\mu_p^s: N_p \to ({N|_T})^s_p$ for non-degenerate rectangles $pqrs$ as above.
The kernels $\Ker(\mu^s_p)$ are subspaces of $N_p$, and if $s'\le s$, then $\Ker(\mu^{s'}_p) \subseteq \Ker(\mu^s_p)$. Since $N_p$ is finite-dimensional, there is some $s$
with $\Ker(\mu^s_p)$ of minimal dimension, and therefore $\Ker(\mu^s_p)$ must be contained in all other kernels.
Now any element $0 \neq \ell \in \Ker(\mu^s_p)$ defines a submodule of $N$ of the form $k_{\{p\}}$, a contradiction.
Thus $\Ker(\mu^s_p) = 0$, so also $\Ker(\mu^{s'}_p) = 0$ for any $s'\le s$. Thus $\mu^{s'}_p$ is an isomorphism for such $s'$,
and so $N_p \to (\overline{N|_T})_p$ is an isomorphism, as desired.
\end{proof}
\bibliographystyle{amsplain}
\bibliography{refs}

\end{document}